\documentclass [11pt] {amsart}
\usepackage{amsthm,amsmath,amssymb,amscd,graphics,enumerate}
\usepackage{latexsym}
\usepackage{verbatim}
\usepackage{enumerate}
\usepackage{amsthm}
\usepackage{amsmath}
\usepackage{amscd}

\newcounter{theorem}[section]
\numberwithin{equation}{section}

\newtheorem{cl}[theorem]{Claim}

\newtheorem{Thm}[theorem]{Theorem}
{\theoremstyle{remark}
 
\newtheorem{Rem}[theorem]{\text{\textbf{Remark}}} }
\newtheorem{Def}[theorem]{Definition}
\newtheorem{Lem}[theorem]{Lemma}
\newtheorem{Prop}[theorem]{Proposition}

\newtheorem{Con}[theorem]{Conjecture}

\newtheorem{notation}[theorem]{Notation}

{\theoremstyle{definition}
 }

\newcommand{\RR}{{\mathbb{R}}}

\newcommand{\CC}{{\mathbb{C}}}
\newcommand{\QQ}{{\mathbb{Q}}}
\newcommand{\NN}{{\mathbb{N}}}
\newcommand{\PP}{{\mathbb{P}}}

\newcommand{\Aut}{{\operatorname{Aut}}}

\newcommand{\mO}{\mathcal{O}}
\newcommand{\mX}{\mathcal{X}}

\newcommand{\mC}{\mathcal{C}}

\newcommand{\mE}{\mathcal{E}}

\newcommand{\mK}{\mathcal{K}}

\newcommand{\mM}{\mathcal{M}}

\newcommand{\mU}{\mathcal{U}}

\newcommand{\fbar}{\bar{f}}

\newcommand{\Fbar}{\bar{F}}
\newcommand{\Gbar}{\bar{G}}
\title{On   Fano varieties with large pseudo-index}
\author{Jiun-Cheng Chen}
\thanks{Research of J.-C. C supported in part by National Center of Theoretical Science in Hsinchu, Taiwan, National Science Council, Taiwan,  and by the Golden Jade Fellowship of Kenda Foundation.}
\address{National Center for Theoretical Sciences, \\Math. Division \\
Third General Building, National Tsing Hua University\\
No. 101 Sec 2 Kuang Fu Road \\Hsinchu, Taiwan 30043, Taiwan}
\email{jcchen@math.nthu.edu, jcchenster@gmail.com}

\date{\today}

\begin{document}
\begin{abstract}
Let $X$ be a Fano variety with at worst isolated quotient singularities.
 Our result asserts that if $C \cdot (-K_X) > max\{\frac{n}{2}+1,\frac{2n}{3}\}$  for every curve $C \subset X$, then $\rho_X=1$. 
\end{abstract}
\maketitle
\markboth{}{}

\section{Introduction}
We work over the field of complex numbers $\CC$. 
Let $X$ be a smooth Fano variety of dimension $n$. 
The index  $r_X$ is defined as   
             $$r_X = max \{m \in \NN | − K_X = mL \;\text{for some line bundle} \; L\}.$$
The pseudo-index $i_X$ is defined as
        $$i_X = min\{m \in \NN | C \cdot (-K_X)  = m \;\text{for some rational curve}\; C \subset X\}.$$
Clearly, we have $i_X \geq r_X$.  Denote by  $\rho_X$  the Picard number of $X$.
In 1988, Mukai \cite{mu88} proposed the following conjecture:
\begin{Con}[Mukai Conjecture]
 Let $X$ be a smooth Fano variety of dimension n. Then $\rho_X \cdot (r_{X} -1) \leq n$.
 \end{Con}
A  generalized version of this conjecture can be stated as follows: 
\begin{Con}[Generalized Mukai Conjecture]\label{genmukai}
Let $X$ be a smooth Fano variety of
dimension $n$. Then (1) $\rho_X \cdot (i_X- 1) \leq n$, and (2) equality holds if and only if
$X \cong  (\PP^{i_{X}} )^{\rho_X}$.
\end{Con}
For each  integer $ k \geq 1$, we can consider the following conjecture.
\begin{Con}[Generalized Mukai Conjecture$\,_k$]\label{subconjecture}
Let $X$ be a smooth Fano variety of
dimension $n$.  
\begin{enumerate}
\item Suppose that $i_X > \frac{n}{k}+1$. Then  $\rho_X  \leq k-1$.
\item Suppose that $i_X \geq \frac{n}{k}+1$   and $\rho_X =k$. Then $X \cong (\PP^{\frac{n}{k}})^k$.
\end{enumerate}
\end{Con}

Clearly,  Conjecture~\ref{subconjecture}  for all $k$ implies  Conjecture~\ref{genmukai}.

When $k=1$, the first part of Conjecture~\ref{subconjecture} implies that if $i_X >n+1$ then $\rho_X \leq 0$. 
Since the 
Picard number of any projective variety is always at least $1$, it implies that $i_X \leq n+1$ for any 
smooth Fano variety of dimension $n$. 
This consequence follows from  Mori's famous bend and break result.  
The second part of Conjecture~\ref{subconjecture} follows from  the characterization result  
of projective spaces by Cho, Miyaoka and Shepherd-Barron \cite{cmsb02} and Kebekus \cite{ke01}.
For  $k=2$,  Wi\'sniewski \cite{wi90} proved that  if $i_X > \frac{n}{2} +1$ then $\rho_X = 1$.

For $k=3$, there are related  results due to Bonavero, Casagrande, Debarre 
and Druel \cite{bcdd02}, and Andreatta,  Chierici, and 
Occhetta \cite{aco03}.
We state their precise  results here.  Bonavero, Casagrande, Debarre and   Druel \cite{bcdd02}  proved  
Conjecture~\ref{genmukai} in the following situations: (1) $X$ has dimension $4$, (2) $X$ is a toric variety
of pseudo-index $i_X  \geq \frac{n}{3} + 1$ or of dimension at most 7. 
Andreatta,  Chierici, and Occhetta \cite{aco03} proved the generalized Mukai conjecture
assuming either  (1) $i_X  \geq \frac{  n}{3} + 1$ and $X$ has a fiber type extremal contraction, 
(2)  $ i_X \geq \frac{ n}{3} + 1$ and $X$ has not small extremal contractions, 
or (3) $dim \;X = 5$.

Now consider the case when $X$ is not smooth.
Even the $k=1$ case, i.e. $i_X \leq n+1$, is not known.     
When $X$ has  only  quotient singularities or isolated LCIQ singularities, Tseng and I proved that $i_X \leq n+1$ \cite{ct05}.  
In the same paper, we also proved if $X$ is  log terminal and $dim\;X=3 $ then $i_X \leq 3+1$. 

In \cite{ct05a}, we proved that (1) if $X$ has at worst isolated LCIQ singularities 
and $i_X \geq n+1$ then $X \cong \PP^n$, and (2)
if $X$ a normal $\QQ$-factorial Fano threefold 
and $i_X \geq 4$ then $X \cong \PP^3$.  In \cite{ch06}, we proved that if $X$ has 
at worst quotient singularities and $i_X \geq n+1$ then $X  \cong \PP^n$.

In this note, we would like to study the $k=2$ case when the variety $X$ has only isolated quotient singularities. 
The following theorem is the main result of this note.      
\begin{Thm}\label{mainthm}
Let $X$ be a Fano variety of dimension $n$ with at worst isolated quotient singularities. 
If  $i_X > max \{\frac{n}{2}+1, \frac{2n}{3} \}$, then   $\rho_X=1$.

\end{Thm}  

Consider a smooth Fano variety $X$.
There are a lot of rational curves on $X$ and this fact is  important in  studying the geometry of $X$. 
  The deformation theory of rational curves is well-known and has many important consequences. 
When $X$ is not smooth, the situation is quite different. 
The dimension of deformations (of rational curves) tends to be smaller.  
It is also possible that  all rational curves pass through a singular point. 
These obstacles make the study of deformations of rational curves on
singular Fano varieties quite difficult.  
When $X$ has only quotient singularities, one natural approach  
is to consider the covering stack $\mX$ and 
 study "rational curves" on the covering stack $\mX$. Pushing forward any 
family of curves on $\mX$  to $X$ yields a 
family of curves on $X$. If we can obtain a nice family of rational curves on $\mX$, we have a nice 
family of rational curves on $X$ for free.  
Let $g: C \to \mX$ be a rational curve. The lower bound  on the dimension of the 
deformation space of $g: C \to \mX$  can be computed as in the smooth case. 
Roughly speaking, the deformation theory does not see the twisted points on the target $\mX$.
So far, everything seems work. But there are several difficulties for this approach.
One  problem is that when we have a curve $g: C \to X$ it may not be possible to 
 lift it to  $\tilde{g}: C \to \mX$. To get a lifting, one has to add twisted points on 
the source curve $C$ \cite{av02}. Therefore, we have to study the deformations of twisted curves into $\mX$. 
Unfortunately, the presence of twisted points on the source curve will lower the dimension of the 
deformation space (due to the age terms, see Lemma~\ref{def} for the precise formula). The second 
bad news is that even if we start with a 
curve (without twisted points) $g: C \to\mX$, we may have to consider twisted points after applying  
bend and break to obtain  curves of smaller degree.  

In \cite{ct05} and \cite{ct05a}, we developed several  techniques to handle the problems caused by the 
presence of twisted points. One of our 
main results in \cite{ct05} is that if we start with a curve (in an extremal ray) without 
twisted point then we can find a twisted curve with at most one twisted point  in the same extremal ray 
and bound the degree at the same time.  
Twisted curves with only one twisted point are good enough in many cases when we try to apply  bend and break 
type arguments.

The   starting point of this work is  a very naive dimension count. 
Suppose that $\rho_X \geq 2$. We can find 
two extremal classes $\alpha$ and $\beta$. Lemma~\ref{1twisted} guarantees that we can find twisted curves
$f_1: \mC_1 \to \mX$ and $f_2: \mC_2 \to \mX$ whose classes are $\alpha$ and $\beta$ respectively. For $i=1, \;2$, the twisted curve $\mC_i$  has at most one twisted point and $ \mC_i \cdot -K_{\mX} \leq n+1$.
For simplicity, assume both  $\mC_1$ and $\mC_2$  do have  one twisted point.  
Let $V_1 \subset \mX$ be  the subset swept by the 
deformation of $f_1(\mC_1) \subset \mX$ and $V_2 \subset \mX$  the subset swept by the 
deformation of $f_2(\mC_2) \subset \mX$. 
Since the twisted points on $\mX$ are isolated, the image of the twisted point on $\mC_1$ does not deform. 
Therefore  every  deformation of $f_1(\mC_1)$ has to pass through the specified point. Similarly, every deformation of 
$f_2( \mC_2)$ also has to pass through another specified point. 
Lemma~\ref{boundpicardnumber2} implies that these two sets should not intersect. 
It is easy to see  $dim V_1 \geq i_X-1$ and $dimV_2 \geq i_X -1$. 
We "expect" them to intersect if $2i_X-2 -n  >0$, i.e. $i_X > \frac{n}{2}+1$. 
Therefore, the assumption $\rho_X \geq 2$ is wrong whenever $i_X > \frac{n}{2}+1$. 
Of course, this "argument" is  very problematic since we can have two  disjoint divisors in a variety. 

To remedy this incomplete argument, we consider  a covering family  $\mC \to \mX$ of twisted stable maps.      
Note that there may be many twisted points  on a general source curve $\mC_t$. 

We only sketch the argument for the worst case, i.e. the source curve $\mC_t$ has at lease two 
twisted points.  In general, the presence of twisted points on the source curve causes troubles.
However, when $X_{sing}$ is isolated, we can play the presence of many twisted points  to our advantage via
bend and break type arguments: the  images of these twisted points can not move;  Mori's bend and break  result  then implies this    
covering family will degenerate somewhere. 

Assume that  $\mC_t$ has at least $2$ twisted points. Using a bend and break type argument, we can  show that there is a family of unsplit curves such that the locus 
of this family contains at least a divisor (see Proposition~\ref{2twistedpoints} for more details).  Denote this divisor  by $D$.  
Once we have such a family of curves, we  find an extremal  contraction, denoted by $\phi: X \to Y$,  whose exceptional set contains the  divisor $D$. We will show that $Y$ is a point. 
This contraction is either a divisorial contraction or a contraction of fiber type.
The contraction $\phi: X \to Y$  has the following property: $C \cdot D >0$ where $C$ is any curve contracted by $\phi: X \to Y$.  

If the contraction $\phi: X \to Y$ is divisorial, then the inequality $C \cdot D >0$ gives a contradiction by a standard argument.

If the contraction is of fiber type and $dim \;Y >0$, then we can find a closed subset $Z \subset X$ such that $dim\; Z> dim\;Y$ and $Z \to Y$ is finite. This is again a contradiction (see Lemma~\ref{divisorlemma} for more details).

  The rest of this paper is organized as follows:
In Section 2, we recall basic definitions on  twisted curves and twisted stable maps. 
We also recall  Lemma 7.1 from \cite{aco03} in Section 2.   
The main theorem is proved in Section 3.

\section*{Acknowledgements}  
The author  likes to
thank Dan Abramovich,  Lawrence Ein and Hsian-Hua Tseng 
for helpful  discussions and valuable suggestions. The author also likes to thank the referee for his valuable comments and suggestions.

\section{Twisted curves} 
We recall some basic facts on twisted curves in this section.  All results in this section are not new; we recall these facts for the reader's 
convenience. 

Twisted curves appear naturally in compactifying stable maps into a proper Deligne-Mumford stack  \cite{av02}.   
Roughly speaking, twisted curves  are nodal curves having  certain stack structures \'etale locally near nodes (and, for pointed curves, marked points).  
For the precise definition, see  \cite{av02}, Definition 4.1.2.

Let $\mC$ be a twisted curve and $C$ its coarse moduli space.
\subsubsection{Nodes}
For a positive integer $r$, let $\mu_r$ denote the cyclic group of $r$-th roots of unity. \'Etale locally near a node, a twisted curve $\mC$ is isomorphic to the stack quotient $[U/\mu_r]$ of the nodal curve $U=\{xy=f(t)\}$ by the following action of $\mu_r$: $$(x,y)\mapsto (\zeta_r x,\zeta_r^{-1} y),$$ where $\zeta_r$ is a primitive $r$-th root of unity. \'Etale locally near this node, the coarse curve $C$ is isomorphic to the schematic quotient $U/\mu_r$.
\subsubsection{Markings}
\'Etale locally near a marked point, $\mC$ is isomorphic to the stack quotient $[U/\mu_r]$. Here $U$ is a smooth curve with local coordinate $z$ defining the marked point, and the $\mu_r$-action is defined by $$z\mapsto \zeta_r z.$$ Near this marked point the coarse curve is the schematic quotient $U/\mu_r$.
\subsection{Twisted stable maps}
\begin{Def}
A twisted $n$-pointed stable map of genus $g$ and degree $d$ over a scheme $S$ consists of the following data (see \cite{av02}, Definition 4.3.1):
$$\begin{CD}
\mC @>f>>  \mX \\
@V{\pi_\mC}VV @V{\pi}VV \\
C @>\fbar>> X \\
@V{}VV \\
S.
\end{CD}$$
along with $n$ closed substacks $\Sigma_i\subset \mC$ such that
\begin{enumerate}
\item $\mC$ is a twisted nodal $n$-pointed curve over $S$ (see \cite{av02}, Definition 4.1.2),
\item $f:\mC\to \mX$ is representable,
\item $\Sigma_i$ is an \'etale gerbe over $S$, for $i=1,...,n$, and
\item the map $\fbar: (C,\{p_i\})\to X$ between coarse moduli spaces induced from $f$ is a stable $n$-pointed map of degree $d$ in the usual sense.
\end{enumerate}
\end{Def}

A twisted map  $f:\mC\to \mX$ is stable if and only if for every irreducible component $\mC_i\subset \mC$, one of the following cases holds:
\begin{enumerate}
\item
$f|_{\mC_i}$ is nonconstant,
\item
$f|_{\mC_i}$ is constant, and $\mC_i$ is of genus at least $2$,
\item
$f|_{\mC_i}$ is constant, $\mC_i$ is of genus $1$, and there is at least one special points on $\mC_i$,
\item
$f|_{\mC_i}$ is constant, $\mC_i$ is of genus $0$, and there are at least three special points on $\mC_i$.
\end{enumerate}
In particular, a nonconstant representable morphism from a smooth twisted curve to $\mX$ is stable.

We say a twisted stable map $\mC \to \mX$ is rational if the coarse moduli space $C$ of $\mC$ is rational.

We will use the following two basic facts on twisted stable maps frequently in this note: \newline 

(1) Let $f:\mC \to \mX$  be a twisted stable map. Then the image of any twisted point on $\mC$ has to be a twisted point on $\mX$. \newline 

(2) Consider a family of rational  twisted stable map  $F:\mC \to \mX$ over an irreducible projective scheme $T$. If this family does not degenerate, i.e.  (i) $\mC_t$  is irreducible for every $t \in T$
  and (ii) $F_t: \mC_t \to \mX$ is birational to its image for every $t$, then the number of twisted points on $\mC_t$ is constant. \newline 

Fix  an ample $\QQ$-line bundle $H$ on $X$. 
Let $\mK_{g,n}(\mX,d)$ be  the category of twisted $n$-pointed stable maps 
to $\mX$ of genus $g$ and degree $d$ with respect to the pull-back of $H$. We will set $H= -K_X$ when $X$ is a Fano variety with at worst quotient singularities. 
It is known  that $\mK_{g,n}(\mX,d)$ is a proper Deligne-Mumford stack 
with projective coarse moduli space denoted by  $K_{g,n}(\mX,d)$ \cite{av02}.
Let $\beta \in H_2(X)$ be a homology class. 
The space of twisted $n$-pointed stable maps $f:\mC \to \mX$ of genus $g$  and  
  homology class $(\pi \circ f)_*[\mC]= \beta$ is denoted by  $\mK_{g,n}(\mX, \beta)$.  This  stack is also proper \cite{av02}.

\subsection{Morphism space  from a twisted curve to a Deligne-Mumford stack}
In this paper, we use both the stack of  twisted stable maps and  
the morphism space from $\mC$ to $\mX$. Roughly speaking, an element in 
the morphism space $Mor(\mC, \mX)$ is a twisted stable map together with a 
parametrization
on the source curve $\mC$.   
Let  $\Sigma \subset \mC$ be the set of twisted points and 
$B \subset \mC$  a finite set of points (twisted or untwisted).
Let $f:\mC \to \mX$ be a representable morphism.  
When $\mX$ is smooth, 
we have a lower bound on  
the dimension of  $Mor(\mC,\mX;f|_B)$ near the morphism $[f]$.
\begin{Lem}[= \cite{ct05} Lemma 4.4] \label{def}
$$dim_{[f]}Mor(\mC,\mX;f|_B)\geq -\mC\cdot K_\mX + n [\chi(\mO_{\mC}) - \;Card(B)]-\sum_{x\in \Sigma\setminus B} age(f^* T_\mX,x).$$
\end{Lem}
\begin{Rem}\label{age}
For each twisted point $x \in \mC$, the contribution from the age term  is strictly less than  $dim \;X=n$ \cite{ct05}.  
\end{Rem}

Consider  $$\begin{CD}
\mC @>F>>  \mX \\
@V{g}VV  \\
T \\
\end{CD}$$
a family of  twisted stable maps.  
We are not only interested in the dimension of $T$ but also the dimension of $F(\mC)$. 
In general, information on the  dimension of $T$ (the parametric space) does not yield much 
information  on  the dimension of $F(\mC)$ (the locus swept by the curves  $\{\mC_{t} | t \in T \}$). 
The main problem is that for any given point $x$ there may be a positive dimensional family of curves passing through $x$. 
For a special class of curves, this problem disappears.
\begin{Def}
Let $C \subset X$ be a curve. We say $C$ is unsplit if  the homology class  $[C] \in H_{2}(X)$  can not be written 
as a non-trivial positive integral sum of curve classes, i.e.  if  $[C]=\sum_{i=1}^{k} a_i[C_i]$ ($a_i \in \NN_{>0}$) then $k=1$ and $[C_1]= [C]$. 
\end{Def}
  
\begin{Def}
Let 
$$\begin{CD}
\mC @>F>>  \mX \\
@V{\pi_\mC}VV @V{\pi}VV \\
C @>\Fbar>> X \\
@V{\Gbar}VV \\
T
\end{CD}$$
be a  family of twisted stable maps. We say it is an unsplit family  if 
$F_t$ is birational to its image and  $(\pi \circ F_t)(\mC_t) \subset X$ is unsplit for every $t \in T$. 
\end{Def}

When the family is unsplit, we have the following bound:

\begin{Lem} \label{dimoflocus}
Let 
$$\begin{CD}
\mC @>F>>  \mX \\
@V{\pi_\mC}VV @V{\pi}VV \\
C @>\Fbar>> X \\
@V{\Gbar}VV \\
T
\end{CD}$$
be an   unsplit family of twisted $1$-pointed stable maps. Assume the image of the marked point is fixed. Let $d= \mC_t \cdot -K_{\mX}$ be the anti-canonical degree.
Then
$$ dim  F(\mC) \geq d -1 $$
\end{Lem}
\begin{proof}
This lemma follows easily from  Mori's bend and break,  Lemma~\ref{def}, Remark~\ref{age} and the fact that $dim \Aut (\PP^1, 0)=2$.  
\end{proof}

\subsection{Covering Stack and Lifting}
We start with the notation of the covering stack $\mX$. 
\begin{notation}\label{stack}
Let $X$ be a normal projective variety with at worst quotient singularities.  
Fix a natural number $r \in \NN_{>0}$ such that $r K_X$ is Cartier.
Fix  a  proper smooth Deligne-Mumford stack $\pi: \mX \to X$ such that $X$ is a coarse moduli space of $\mX$ and  $\pi$ is 
an isomorphism over $X_{reg}=X \setminus X_{sing}$.  
Note that  $K_{\mX}=\pi^{*}K_X$ and $$\mC \cdot K_{\mX}= C
\cdot K_X$$ for any (twisted) curve $\mC \to \mX$ with coarse curve $C$
\cite{ct05}.
\end{notation}

Let $C$ be a smooth irreducible curve  and $\fbar: C \rightarrow X$ a morphism.
In general, it is not possible 
to ``lift'  the map $\fbar:C \to X$ to a map
$C \to \mX$.  However, we can endow an orbicurve structure on $C$  and lift $C \to X$ to $\mC \to \mX$  (by Lemma 7.2.5 \cite{av02}).

An arbitrary lifting is not very useful since the dimension lower bound obtained from Lemma~\ref{def} will be too small. One would like 
to have a lifting with as few twisted points as possible.  We do not know how to achieve this goal in general. 
If  the class of $C \subset X$ is in an extremal  ray,  we are able to find a nice lifting as in the next  lemma.

\begin{Lem}[Proposition 3.2 \cite{ct05a}]\label{1twisted}
Notation as in Notation~\ref{stack}.
Let $R \subset \overline{NE}(X)$ be any  $K_X$-negative  extremal ray. Then there exists a twisted rational
 curve $f:\mC \to \mX$ 
such that (1) $\mC$ has  at most one twisted point, (2) the intersection number  $\mC \cdot (-K_{\mX}) \leq n+1$, and (3) $(\pi \circ f)_*[\mC] \in R$.
\end{Lem}
\subsection{Twisted curves with at most 1 twisted point}\label{1twistedcurve}
Assume that $X$ is a normal projective variety of dimension $n$ with at worst isolated quotient singularities.
Let $f: \mC \to \mX$ be an unsplit twisted stable map of genus $0$ and homology class $\alpha \in H_2(X)$. Assume that the 
source curve $\mC$  has at most  $1$ twisted point.

If $\mC $ has one  twisted point, denoted by $\infty \in \mC$, 
take an irreducible component 
$\mM \subset \mK_{0,1}(\mX, \alpha)$ which contains $[f]$. Take the universal family of twisted stable maps over $\mM$. 
Note that the image of $\infty$ is fixed since  $X$ has only isolated singularities.

If $\mC$ does not have any twisted point, i.e. $\mC \cong \PP^1$, pick any point, denoted by $\infty $ again. Consider the curve $f: \mC \to \mX$ with the marked point $\infty \in \mC$  as an 
element of $\mK_{0,1}(\mX, \alpha) $. We  abuse the notation and still denote this 
element in $\mK_{0,1}(\mX, \alpha)$ by $[f]$. Consider an irreducible component $\mM \subset \mK_{0,1}(\mX, \alpha , f|_{\infty})$ which contains $[f]$. Take the universal family of twisted stable maps over $\mM$.

In either case, we obtain a family of twisted stable 
maps such that  the image of the marked point is fixed.

\subsection{Bounding Picard Number}
Let $X$ be a normal projective variety and 
 $$\begin{CD}
C @>F>>  X \\
@V{G}VV  \\
V \\
\end{CD}$$ be a family of stable maps of genus $0$  over an irreducible scheme  $V$. 

We denote by $Locus (V)$ the image of $C$ in $X$. 
Let $p \in X$ be a (closed) point.  
Define   $Locus(V)_p:=  F( G^{-1}(G( F^{-1}(p))))$, i.e.  the set of points  $x \in X$ such that there is a $v_0 \in V$ satisfying $x \in F(C_{v_0})$ and 
$p \in F(C_{v_0})$.

The following lemma is a special case of Lemma 5.1 in \cite{aco03}.

\begin{Lem}\label{boundpicardnumber2}
Let $X$ be a normal projective variety,  $p \in X$  any point, 
 and $V$ 
an unsplit family of
rational curves. Also assume that $V$ is irreducible and projective.
Then $Locus(V )_p$ is closed and every curve contained in $Locus(V )_p$
is numerically equivalent to   $\mu  C_{V}$,
where  $C_{V}$ is the class of $F(C_{v})$ for  $v \in V$ general and $\mu \geq 0$. 
\end{Lem}

\subsection{Covering family of rational curves}
The following lemma is an easy consequence of Kawamata's famous result \cite{ka91}:
\begin{Lem}\label{cover1}
Let $X$ be a Fano variety of dimension $n$ with at worst quotient singularities. Then   there is a covering family of rational curves with $-K_X$-degree at most $2n$.
\end{Lem}

We have the following  stack version of Lemma~\ref{cover1}.
\begin{Lem}\label{cover2}
Let $X$ be a Fano variety of dimension $n$. Assume that $X$ has at worst quotient singularities.  Let $\mX \to X$ be the smooth covering
stack as in Notation~\ref{stack}. Then  there is a covering family of twisted rational curves  with $-K_{\mX}$-degree at most $2n$.
\end{Lem}
\begin{proof}
Let $r$ be the natural number such that $r K_X$ is Cartier as in Notation~\ref{stack} and  
$d>0$ be any rational number such that $r d \in \NN_{>0}$. 
Consider the morphisms $\coprod _{k \in \NN} \mK_{0,k}(\mX, d) \to \coprod _{k \in \NN} \mK_{0,k}(X, d)$ (forgetting the stack structures) 
and  $\coprod _{k \in \NN} \mK_{0,k}(X, d) \to \mK_{0,\;0}(X,d)$ (forgetting marked points and then  stabilizing).
Note that the composition  $\coprod _{k \in \NN} \mK_{0,k}(\mX, d) \to   \mK_{0,\;0}(X,d)$ is surjective since  
we can always lift a stable map $C \to X$ to $\mX$  by adding suitable stack structures 
on finitely many points on $C$. Consider the universal family $F_{k,\;d}: \mU_{k,\;d} \to \mX$ over  $\mK_{0,k}(\mX, d)$. 
It follows  that $$\cup_{\{(k,d)|k \in \NN,\; d \leq 2n,\; r d \in \NN\}}(\pi \circ F_{k, \;d}) (\mU_{k,\;d}) = X$$
 since $X$ is covered by rational curves with degree at most $2n$. Since $X$ is Noetherian and irreducible,  we 
have  $(\pi \circ F_{k,\;d}) (\mU_{k,\;d})=X $ for some $(k,d)$.

Since $\mK_{0,k}(\mX,d)$ is a proper Deligne-Mumford stack with projective coarse moduli space, 
we can find an \'etale surjective morphism $V \to \mK_{0,k}(\mX,d)$ such that $V$ is projective.

Pulling back  the universal family of twisted stable maps
over $\mK_{0,k}(\mX,d)$ to $V$  yields  a covering family of twisted stable maps over $V$.
 
Taking a suitable irreducible component of $V$, we obtain a covering family over an irreducible projective scheme.

\end{proof}

\section{Proof of Theorem~\ref{mainthm}}
We will assume that  $X$ is a Fano variety of dimension $n$ with at worst isolated quotient singularities throughout  this section. 
Let $\pi: \mX \to X$ be the covering stack of $X$ as in Notation~\ref{stack}. 

We start with the following important  lemma. 
\begin{Lem}\label{divisorlemma} 
Consider 
$$\begin{CD}
\mC @>F>>  \mX \\
@V{\pi_\mC}VV @V{\pi}VV \\
C @>\Fbar>> X \\
@V{\Gbar}VV \\
T
\end{CD}$$
 an unsplit  family of   twisted stable maps over an irreducible projective scheme  $T$. 
 Assume that (1) the source curve $\mC_t$  has at most $1$ twisted point for a general $t \in T$, (2)
the image $\Fbar (C)$  
 contains a divisor in $X$ and (3)  $i_X > \frac{n}{2}+1$.  Then  $\rho_X=1$.
\end{Lem}

\begin{proof}
First note that  $\mC_t$  has at most $1$ twisted point for any $t \in T$  since the family is unsplit. 
Consider  the closed subset $V= \Fbar(C)$, the image of the family of twisted stable maps in $X$.
Since $V$ is the image of an irreducible scheme under a projective morphism, it is closed and irreducible. Endow the reduced 
scheme structure on $V$. 
If $V=X$, take $D$ be any irreducible divisor on $X$.
If $V$ is a divisor, take $D$ be any prime divisor contained in $V$.

 Pick a general $t \in T$ and consider the twisted stable map $F_{t}: \mC_{t} \to \mX$.  
Set $\beta=(\pi \circ F_{t})_{\; *} [\mC_{t}]=\Fbar_{t \; *}[C_{t}] \in H_2(X)$.

Recall that $X$ is $\QQ$-factorial. Let $E$ be any curve which intersects $D$ but is  not contained in $D$. It is clear that $E \cdot D >0$.
Write $[E]= \sum a_i [E_i]  $ as a positive combination of extremal rays.  We have $E_i \cdot D >0$  for some $i$. May assume that $E_1 \cdot D >0$. 
Let $\alpha:= [E_1]$.    
By Lemma~\ref{1twisted}, we can find a twisted rational curve $f: \mE \to \mX$ with at most one twisted point such 
that $(\pi \circ f)_{*}[\mE] \in \RR_{\geq 0}\; \alpha$ 
and $\mE \cdot f^{*}(-K_{\mX}) \leq n+1$.  Note that $f(\mE)$ is unsplit since $i_X >\frac{n}{2}+1$.

Note that $\mE \cdot (\pi \circ f)^{*} D >0$ and therefore $f(\mE)$ meets $\pi^{-1}(D)$. 
Replacing   $\alpha$ by  the class $(\pi \circ f) (\mE)$   if necessary (they are two classes in the same ray), we may assume that $[(\pi \circ f) (\mE)]=\alpha \in H_2(X)$.

As in Section~\ref{1twistedcurve}, we can view $f: \mE \to \mX$  as an element of $\mK_{0,1}(\mX, \alpha)$.  Consider  an irreducible component $[f] \in \mM \subset \mK_{0,1}(\mX, \alpha)$ and its universal family of twisted $1$-pointed stable maps  as in Section~\ref{1twistedcurve}. 
Denote by $W_1$ the image of 
this family in $X$. 
 Note that for this family of 
twisted $1$-pointed stable maps the image of 
the marked point is fixed.   

By Lemma~\ref{boundpicardnumber2}, it follows that for any curve 
$C \subset W_1$, its class $[C]$ lies in $\RR_{\geq 0} \;\alpha$. We also 
have $dim\; W_1 \geq \mE \cdot (-K_{\mX}) -1 >\frac{n}{2}$ by Lemma~\ref{dimoflocus}.  

We proceed by showing that the classes $\alpha$ and $\beta$ lie in the same extremal ray. 
We then  show that $\RR_{\geq 0} \; \alpha$ is the only extremal ray.

Divide into two cases: (1) the source curve $\mC_t$ for a general $t \in T$ does have a twisted point, 
and (2) the source curve $\mC_t$  for a general $t \in T$ does not have any twisted point, i.e. $\mC =C$ is a scheme. \newline

Case (1): The source curve $\mC_t$  has one twisted point.

Note that $D$ is covered by a family of rational curves  through a specified point (the image of the twisted point). By Lemma~\ref{boundpicardnumber2} again, it follows that for any curve 
$C \subset D$, its class $[C]$ lies in $\RR_{\geq 0}\; \beta$. Consider $W_1 \cap D$. Since the intersection is non-empty, any irreducible component of $W_1 \cap D$  
has dimension at least $dim W_1 + dim D -n \geq 1$. Therefore, we can find a  curve $C_1 \subset D$ and $C_1 \subset W_1$. 
It follows that $\RR_{\geq 0} \; \alpha =\RR_{\geq 0} \; \beta$. \newline

Case (2): The source curve $\mC_t$ has no twisted point.

Pick any $t_0 \in T$ such that $W_1$ meets  $(\pi \circ F_{t_0})(\mC_{t_0})$. Pick any point on $\mC_{t_0}$, 
say  $\infty  $, and consider the morphism $F_{t_0}: \mC_{t_{0}} \to \mX$ with the marked point $\infty  \in \mC_{t_0}$ as a twisted $1$-pointed stable 
map as in Section~\ref{1twistedcurve}. We  still denote this element by $[F_{t_0}]$.  
Let  $\mM \subset \mK_{0,1} (\mX, \beta, F_{t_0}|_{\infty})$ be an irreducible component which contains 
$[F_{t_0}]$. Consider the universal family of twisted $1$-pointed stable maps over $\mM$ and the image  of this family in $X$. 
Denote the image  by $W_2$. 
Note that for any curve $C \subset W_2$, we have  $[C] \in \RR_{\geq 0} \; \beta$  by Lemma~\ref{boundpicardnumber2}.
By Lemma~\ref{dimoflocus}, we have  $dim W_2  \geq \mC_{t_{0}}  \cdot F_{t_0}^{*} (-K_{\mX})-1 >\frac{n}{2}$.    
Since $W_1 \cap W_2$ is non-empty,  it follows that
$$dim (W_1 \cap W_2) \geq  dim W_{1}+ dim W_{2} -n  \geq 1.$$ Therefore we can find a curve 
$C_1 \subset W_1 \cap W_2$. Again, we have  $\RR_{\geq 0} \; \alpha =\RR_{\geq 0} \; \beta$.
\newline

Consider the extremal ray $\RR_{\geq 0}\; \alpha$ and the corresponding extremal  contraction $\phi_{\alpha}:X \to Y$.
We will prove that $Y$ is a point (and hence $\rho_X=1$).

Since  $\RR_{\geq 0} \;\alpha =\RR_{\geq 0} \;\beta$, for every $t \in T$ the curve $(\pi \circ F_{t})(\mC_{t})$  is contracted by $\phi_{\alpha}$. Therefore $D \subset Exc(\phi_{\alpha})$. The morphism $\phi_{\alpha}$ is either a divisorial contraction or a contraction of fiber type.  
\newline

We first show that $\phi_{\alpha}: X \to Y$ can not be divisorial:
Assume the contrary. Let $A$ be a very ample divisor on $Y$ and $M$ a very ample divisor on $X$.
Set $$C':= D \cap_{i=1}^{dim \phi_{\alpha}(D)} L_i \cap_{j=dim \phi_{\alpha}(D) +1}^{n-2} M_j$$ where
$L_i$'s are the general members of $|\phi_{\alpha}^{*} A|$  and $M_j$'s the general members of $|M|$.  
Let $S= \cap_{i=1}^{dim \phi_{\alpha}(D)} L_i \cap_{j=dim \phi_{\alpha}(D) +1}^{n-2} M_j$. 
Note that $C'$ is contracted by $\phi_{\alpha}$. 
Therefore, it is a (positive) multiple of $\alpha$. 
Note that $C' \cdot D=  (C')^{2}_{S} <0$ since $C'$ is a contracted curve on the surface $S$. 
Thus $\alpha \cdot D <0$. This gives a contradiction. \newline

Now we show that $Y$ is a point.  
Assume by contradiction that  $dim \;Y>0$. 
Let $y \in Y$ be a general point. The fiber $X_y$  is a smooth Fano variety. Note that 
$K_X|_{X_{y}}$ is numerically equivalent to $K_{X_{y}}$. 
By Mori's bend and break, there is a curve $C' \subset X_y$ such that $$C' \cdot (-K_{X_{y}}) \leq dim X_y +1.$$

Note that  $$C' \cdot ( -K_{X_{y}})= C' \cdot (- K_{X} )>  \frac{n}{2}+1$$ by our assumption. 
Combining these  two inequalities yields  $dim X_y >  \frac{n}{2}$.
Thus $dim Y <  n- \frac{n}{2}=\frac{n}{2} $.

Let $\RR_{\geq 0} \; \gamma$ be an extremal ray which is not contracted by $\phi_{\alpha}$.
By Lemma~\ref{1twisted}, there is a twisted stable map $g: \mE' \to \mX$ such that (1) $\mE'$  has at most $1$ twisted point, (2) the intersection number  $\mE' \cdot (-K_{\mX}) \leq n+1$, and (3) $(\pi \circ g)_*[\mE'] \in \RR_{\geq 0} \; \gamma$. We may assume $(\pi \circ g)_*[\mE']= \gamma$.

Again, we  view $g:  \mE' \to \mX$ as an element of $\mK_{0,1}(\mX, \gamma)$ as in Section~\ref{1twistedcurve}.
Take an irreducible component  $ \mM \subset \mK_{0,1}(\mX, \gamma)$ which contains $[g]$ and consider the universal family of twisted stable maps  over $\mM$ as in
Section~\ref{1twistedcurve}. Denote by $W_3$ the image of 
this family in $X$. 
For this family of twisted $1$-pointed stable maps, the image of 
the marked point $\infty$ is also fixed.  
%

Note that if $C'' \subset W_3$ is any curve, then its class is a multiple of $\gamma$ by Lemma~\ref{boundpicardnumber2}. 
By Lemma~\ref{dimoflocus}, we have $dim\; W_3 >  \frac{n}{2}$ and hence $dim\; W_3 > dim\; Y$. 
For any $y \in Y$,  consider the intersection $W_3 \cap X_y $. 
It is easy to see that that $dim\;  W_3 \cap X_y <1$; otherwise, we would 
have $\RR_{\geq 0} \; \alpha= \RR_{\geq 0} \; \gamma$.

It follows that the morphism  $\phi_{\alpha}|_{W_{3}}: W_3 \to Y$ is finite to its image. This is absurd since $dim \; W_3 > dim\; Y$.

\end{proof}

\begin{notation}\label{coveringfamily}
Let  $$\begin{CD}
\mC @>F>>  \mX \\
@V{G}VV  \\
T \\
\end{CD}$$
be a covering family of twisted stable maps over an irreducible and projective scheme $T$. 
Pushing forward to the coarse spaces, we have 
$$\begin{CD}
C @>\Fbar>>  X \\
@V{\Gbar}VV  \\
T \\
\end{CD}$$
a covering family of stable maps into $X$.
Let $k$ be the number of twisted points on the source curve   $\mC_t$   for a general $t \in T$.  
Let $d= \mC_{t} \cdot F_{t}^*( -K_{\mX})$.   
By Lemma~\ref{cover2}, we may assume that $d \leq 2n$.  Choose a covering family with the smallest possible anti-canonical degree. 
We fix this  covering family of anti-canonical degree $d  \leq 2n$ in  Proposition~\ref{2twistedpoints} 
and Proposition~\ref{1point}. 
 
\end{notation}

We proceed by dividing into two cases according to the number of twisted points on a general source curve: (1) $k \geq 2$, and (2) $k=0$ or $1$. 
Proposition~\ref{2twistedpoints} takes care of the first case and   Proposition~\ref{1point} takes care of the second. 
\newline

Case (1): $k \geq 2$.

\begin{Prop}\label{2twistedpoints} 
Let $F: \mC \to \mX$ be the covering family of twisted stable maps as in Notation~\ref{coveringfamily}.
Assume that $\mC_t$ has at least two twisted points for a general  $t \in T$. Also assume that $i_X > max \{\frac{n}{2}+1, \frac{2n}{3}\}$. 
Then  $\rho_X=1$.
\end{Prop}
\begin{Rem}
The condition $i_X > \frac{2n}{3}$
 implies that  the domain curve $\mC_t$  has at 
most two components which are not contracted. 
\end{Rem}

\begin{proof}
Let $S \subset T$ be any projective curve. Consider the  pull-back family of twisted stable 
maps $F|_{S}: \mC  | _{S} \to \mX$.  
Suppose that the image of this family in $\mX$ is two dimensional, i.e. the twisted curve deforms in $\mX$. 
Recall that the images of twisted points are fixed in this family. 
By Mori's bend and break, 
this family  has to  degenerate.

 Let $\tilde{D} \subset T$ be the loci where the domain curve is not irreducible. 
Consider $(\pi \circ F) ( G^{-1}(\tilde{D})) \subset X$.
 
\begin{cl}\label{codimension}
The closed subset $(\pi \circ  F) ( G^{-1}(\tilde{D}))$ has at least one component of codimension at most $1$.
\end{cl}
\begin{proof}[Proof of Claim~\ref{codimension}]
Assume by contradiction that every component of  $(\pi  \circ F) ( G^{-1}(\tilde{D}))$ has codimension at least $2$. 
 
We can  find a curve $E \subset X$ such that $E$ does not meet $(\pi \circ F)( G^{-1}(\tilde{D}))$ and $E$ is not contained in the image of
any fiber of $\mC \to T$ (by taking $E$ the intersection of general ample divisors). Consider $G((\pi \circ F)^{-1}(E)) \subset T$. We can find a (complete) curve $E' \subset  G^{-1}(G((\pi \circ F)^{-1}(E)))$ 
such that $F(E')$ is not contained in the image of any fiber of $\mC \to T$. Take $E'':= G(E')$
 and consider the family  $F|_{E''}: \mC_{E''} \to \mX$. Note that the image of $F|_{E''}: \mC_{E''} \to \mX$ is two dimensional.  By Mori's bend and break, this family has to degenerate, and hence $E'$ has to meet $\tilde{D}$. This gives a contradiction.
\end{proof}

Continue the proof of Proposition~\ref{2twistedpoints}. 
Let $D$ be  an irreducible divisor contained in  $ (\pi \circ F) (G^{-1}(\tilde{D}))$. 
Consider  possible degenerations of twisted  stable maps. By degree reason, only two components 
 of $\mC_t$ are not contracted by $F_t$. Denote these two components by  $\mC^1_t$ and $\mC^2_t$.  Note that  $(\pi \circ F_t)(\mC^1_t)$ and $(\pi \circ  F_t)(\mC^2_t)$ are unsplit.   
If $\mC^i_t , \; i=1,2,$ has two or more twisted points, then its image in $\mX$ can not deform by Mori's bend and break argument. 

Since the locus of degenerate curves in $X$ contains at least a divisor, it follows that at 
least one component of the domain curve, say $\mC^1_{t }$, will have at most one twisted point. 
Thus there is a family of twisted stable maps with at most one twisted point and its image in $X$ contains at least a divisor.
We apply Lemma~\ref{divisorlemma} and  conclude the proof of Proposition~\ref{2twistedpoints}.
\end{proof}

Case (2): $k \leq 1$.

\begin{Prop}\label{1point}
Let  $F: \mC \to \mX$  be the covering family as in Notation~\ref{coveringfamily}.
Assume that $\mC_t$ has at most one twisted point for a general  $t \in T$.
Also assume that $i_X > max \{ \frac{n}{2}+1, \frac{2n}{3} \}$. Then $\rho_{X}=1$.
\end{Prop}
\begin{proof}

We first prove the following claim:
\begin{cl}\label{lowerdegree}
For  a very general point  $x \in \mX$, there is a twisted rational curve $\mE$ through $x$ such that (1) $\mE$ has at most 
one  twisted point, and (2) $\mE$ has  $-K_{\mX}$-degree at most $n+1$.
\end{cl}

\begin{proof}[Proof of Claim~\ref{lowerdegree}]

Note that there are only countably many unsplit classes $\alpha_{i} \in H_{2}(X), \; i \in I$. 
Consider all possible twisted stable maps such that (1) the source curve is birational to its  image, (2) the source curve has  at least two twisted points,  and (3) the class of the image is one of these $\alpha_{i} $'s. 
For each $\alpha_{i}$ there are only finitely many such twisted stable maps thanks to Mori's bend and break.  
Delete the images of all such twisted stable maps and all twisted points from $\mX$. Denote the resulting  set 
by $U$. 

Let $x \in U$.
Pick  a $t \in T$ such that $x$ lies on the image of $\mC_t$ (under the morphism $F_t$). 
Note that  $x$ is the image of an ordinary point, denoted by $0 \in \mC_t$, since the 
image  of a twisted point  is a twisted point. 
If $\mC_t$ has a twisted point, we denote it by $\infty$. 
 If $\mC$ does not have any twisted point, 
pick any point on $\mC_t -\{0\}$  whose image in $\mX$ is  not $x$. Also denote this point  by $\infty \in \mC_t$.

Suppose that $\mC_t  \cdot  F_{t}^{*}(-K_{\mX}) >n+1$. 
Then  
$$dim_{[F_{t}]} Mor (\mC_{t}, \mX, F_{t}|_{\{ 0,  \infty\}}) >1 $$ 
by Lemma~\ref{def}.

By Mori's bend and break, the family $F: \mC \to \mX$ degenerates at some $t' \in T$.  
By degree reason, there are  only two components of $\mC_{t'}$, say $\mC^1_{t'}$ and $\mC^2_{t'}$, which
 are not contracted by $F_{t'}$.  Note that since $F_{t'}$ is a twisted stable map, every contracted component
needs to have at least three special points (nodes or the original twisted point). It implies that the domain curve $\mC_{t'}$ has either two or three components.

In the first case, the two components $\mC^1_{t'}$ and $\mC^2_{t'}$ intersect at a node which may be a new twisted point.  In the second case, there are three components  $\mC^1_{t'}$,  $\mC^2_{t'}$ and $\mC^3_{t'}$. The contracted component  $\mC^3_{t'}$ has three special points: the original twisted point $\infty$ and 
two nodes $p_1= \mC^1_{t'} \cap \mC^3_{t'}$   and $p_2= \mC^2_{t'} \cap \mC^3_{t'}$. Note that $p_1$ and $p_2$ may be new twisted points.   
By this analysis, we can obtain a curve, say  $F_{t'}:\mC^1_{t'} \to \mX$, such that the  
 image $ F_{t'}(\mC^1_{t'})$ contains  $x$ and  the source curve $\mC^1_{t'}$ has at most two twisted points (the original twisted point $\infty$ and the possibly new twisted point at the node).   
We claim  that $\mC^1_{t'}$ can not have two twisted points: Suppose that $\mC^{1}_{t'}$ does have two twisted points.  
The image of $\mC^1_{t'}$ (under $\pi \circ F_{t'}$) will  be  in $X -U$. This is 
not possible since $x \in U$.

Therefore we obtain a twisted rational curve through $x$ with at most $1$ twisted point and of smaller anti-canonical degree.
Repeat this bend and break process until the anti-canonical degree 
is at most  $n+1$. This concludes the proof of Claim~\ref{lowerdegree}.
\end{proof}

Now we show that the  covering family $F:\mC \to \mX$ has anti-canonical degree $d \leq n+1$.

Assume the contrary, i.e.  $d >n+1$.  
Let $r \in \NN$ be the  natural number such that $r K_X$ is a Cartier divisor as in Notation~\ref{stack}. Note that  $r (C \cdot -K_X) \in \NN_{>0}$ for any curve $C \subset X$. Let  $l$ be any rational number such that $0< l \leq n+1$ and $r l \in \NN_{>0}$.
Consider $\mK_{0,1}(\mX, l)$  and its  universal  family of twisted  stable maps $F_l: \mU_l \to \mX$.   
The  union $\cup _{\{l| 0< l \leq n+1, \; rl \in \NN \} } (\pi \circ F_l)(\mU_l)$ 
is a closed set since it is the  union of finitely many closed set. It is clear that $\cup _{\{l| 0< l \leq n+1, \; rl \in \NN \} } (\pi \circ F_l)(\mU_l)=X$  since through any very general point $x \in \mX$ there is a twisted rational curve with at 
most $1$ twisted point and of anti-canonical degree $l \leq n+1 <d$ by  Claim~\ref{lowerdegree}. 
Since $X$ is Noetherian and irreducible, we have $(\pi \circ F_l)(\mU_l)=X$ for some $l$.  
By the same argument  in Lemma~\ref{cover2}, we have a covering family of degree strictly less than $d$ over an irreducible and projective scheme.  This is not possible since our choice of 
the covering family  $F: \mC \to \mX$ is of the smallest (anti-canonical) degree.  
Thus we have  $d \leq n+1$.

Note that $d \leq n+1$ implies that the covering family is unsplit by degree reason.
Now apply Lemma~\ref{divisorlemma} to conclude the proof.
\end{proof}

\begin{proof}[Proof of Theorem~\ref{mainthm}]
Theorem~\ref{mainthm} now follows easily from Proposition~\ref{2twistedpoints} and Proposition~\ref{1point}.
\end{proof}

The statement $$i_X > \frac{n}{2}+1 \Rightarrow \rho_X=1$$ is equivalent to 
  $$\rho_X \geq 2  \Rightarrow i_X \leq \frac{n}{2}+1.$$ Under an extra assumption, we are able to 
prove the following: 

\begin{Prop}
Let $X$ be a Fano variety with at worst isolated quotient singularities. Assume that $\rho_X \geq 2$ and there is an extremal 
contraction  $\phi: X \to Y$ such that $dim\;Exc(\phi) \geq n-1$.  Then $i_X \leq \frac{n}{2}+1$.

\end{Prop}
\begin{proof}
Assume the contrary that $i_X > \frac{n}{2}+1$. Divide into two cases: (1) the contraction $\phi:X \to Y$ is divisorial, and (2) the contraction is of fiber type. 
\newline

First consider the divisorial case. Let $D=Exc(\phi) \subset X $  be the exceptional divisor, and $\RR _{\geq 0}\; \alpha$
be the contracted extremal ray. As in the proof of Lemma~\ref{divisorlemma}, we can find an extremal ray $\RR_{\geq 0}\; \beta$ such that $\beta \cdot D >0$. Let  $g:\mC \to \mX$ be a twisted stable map such that $\mC$ has at most one twisted point and the twisted stable map   
$g: \mC \to \mX$ is birational to its image. We may also assume $[(\pi \circ g) (\mC)]=\beta$.  Consider $g: \mC \to \mX$ as a 
twisted  $1$-pointed stable map of genus $0$  with homology class $\beta$. Let $[g] \in \mM \subset \mK_{0,1}(\mX, \beta , g|_{\infty})$ be an irreducible component. Consider its universal family of twisted $1$-pointed stable maps $G: \mU \to \mX$. Denote by $W \subset X$ the image of this family in $X$. Note that  $dim W > \frac{n}{2}$ (by Lemma~\ref{dimoflocus}).  
Take an irreducible component of $W' \subset W \cap D$. We have  $dim W' > \frac{n}{2}+ (n-1)-n=\frac{n}{2}-1$.  Note that  for any curve $C' \subset W' \subset W$, its class $[C']$ lies in the ray $\RR_{\geq 0} \; \beta$ (by Lemma~\ref{boundpicardnumber2}). 

Set $Z=\phi(D)$ (with the reduced scheme structure). Consider $\phi|_{D}: D \to \phi(D)=Z$.  
By a similar argument  (using intersection number) as in the proof of  Lemma~\ref{divisorlemma}, we 
have  $\RR _{\geq 0} \;\beta \neq  \RR_{\geq 0} \;\alpha$.  

It is easy to see that $dim \;Z >0$: If  $dim \; Z=0$, then $\RR_{\geq 0} \; \alpha = \RR_{\geq 0} \; \beta$; a contradiction.  Pick a general $y \in Z \subset Y$ and consider the fiber $X_y$. Since $X$ has only isolated singularities, we may assume that $X_y$ is 
smooth (and the morphism $\phi$ is smooth near $X_y$).
Since the exceptional divisor $D$ is covered by rational curves in the ray $\RR_{\geq 0} \alpha$, there 
is a (twisted) stable map   $h:\PP^1 \to  X_y \subset \mX$ such that $ [(\pi \circ h)(\PP^1)] \in \RR_{\geq 0} \alpha$. 
We may assume that $h: \PP^1 \to X_y \subset \mX$ is birational to its image.
  If $\PP^1 \cdot h^* (-K_X) >n+1$,  
it is easy to  see that $dim \;Mor_{[h]}(\PP^1, \mX, h|_{\{0, \infty \}})>1$. Applying bend and break, we obtain a rational curve with smaller degree. 
Note that the resulting  rational curve is still on $X_y$ and its class is a multiple of $\alpha$ (since we start with a curve whose class is in the extremal ray $\RR_{\geq 0} \; \alpha$). Continuing this bend and break process, we may assume that $h: \PP^1 \to X_y \subset \mX$  has $-K_{\mX}$-degree at most $n+1$. Hence it is unsplit. 

As before,  we can find an irreducible component $[h] \in \mM_1 \subset \mK_{0,1}(\mX, \alpha)$ and its universal family of twisted $1$-pointed stable maps $H: \mU_1 \to \mX$. Denote by $W_1$ its image in $X$. Note that 
$W_1 \subset X_y$.   Therefore, we have 
$dim X_y \geq dim W_1 > \frac{n}{2}$.  Also note that $\phi |_{W'}: W'  \to Z$ is finite to its image. Therefore, we have $dim Z  \geq  dim W'$. It follows that
$dim D =dim X_y+ dim Z> \frac{n}{2}+ \frac{n}{2}-1=n-1$. This gives a contradiction. \newline

Now consider the fiber type contraction case. The proof is quite similar to the proof of the fiber type case in Lemma~\ref{divisorlemma}. 
Let $\RR_{\geq 0} \beta$ be an extremal ray not contracted by the morphism $\phi$. We can find a twisted $1$-pointed  stable map $h: \mC \to \mX$ of genus $0$ and  homology class
$\beta$. Let $[h] \in \mM_2 \subset \mK_{0,1}(\mX, \beta, h|_{\infty})$ be an irreducible component and $H: \mU_2 \to \mX$ its universal family of twisted $1$-pointed stable maps. Denote by  $W_2$ the image of $\mU_2$  in $X$. Let $y \in Y$ be a general point and $X_y$ be the fiber. 
As in the proof of Lemma~\ref{divisorlemma}, 
we have $dim X_y > \frac{n}{2}$ and  $dim Y = dim X -dim X_y < \frac{n}{2}$. 
Note that $dim W_2 > \frac{n}{2}$ (Lemma~\ref{dimoflocus})  and the morphism $\phi |_{W_{2}}: W_2 \to Y$ is finite to its image. Hence $dim Y \geq dim W_2 > \frac{n}{2}$. This gives a contradiction.

\end{proof}

\end{document}